\documentclass[reqno, a4paper]{amsart}
\usepackage{amsmath, amsthm, amscd, amsfonts, amssymb, graphicx, color, a4wide}

\usepackage{mathrsfs}

\newtheorem{theorem}{Theorem}[section]

\newtheorem{lemma}[theorem]{Lemma}

\theoremstyle{definition}

\def\bb1{{1\hspace*{-2.4pt}\rm{l}}}

\def\beq{\begin{equation}}
\def\eeq{\end{equation}}

\numberwithin{equation}{section}

\newcommand{\R}{{\mathbb R}}

\newcommand{\Z}{{\mathbb Z}}
\newcommand{\C}{{\mathbb C}}
\newcommand{\Fl}{{\mathfrak f}}

\begin{document}

\vspace{0.5cm}

\title[A Beals criterion for magnetic P.D.O proved with magnetic Gabor frames]{A Beals criterion for magnetic pseudodifferential operators proved with magnetic Gabor frames}

\author[H.D.~Cornean]{Horia D. Cornean}
\address[H.D.~Cornean]{Department of Mathematical Sciences, Aalborg University \\ Skjernvej 4A, 9220 Aalborg, Denmark.}
\email{cornean@math.aau.dk}

\author[B.~Helffer]{Bernard Helffer}
\address[B.~Helffer]{Laboratoire
de Math\'ematiques Jean Leray, Universit\'{e} de Nantes and CNRS, 2 rue de la Houssini\`ere 44322 Nantex Cedex (France) and Laboratoire de Math\'ematiques d'Orsay, Univ.
Paris-Sud, Universit\'e Paris-Saclay, France.}
\email{bernard.helffer@univ-nantes.fr}

\author[R.~Purice]{Radu Purice}
\address[R.~Purice]{Institute
of Mathematics Simion Stoilow of the Romanian Academy \& Centre Francophone en Math\'{e}matiques de Bucarest, P.O.  Box
1-764, Bucharest, RO-014700, Romania.}
\email{radu.purice@imar.ro}

\begin{abstract}
 
 First, we give a new proof for the Beals commutator criterion for non-magnetic Weyl pseudo-differential operators based on classical Gabor tight frames. Second, by introducing a modified 'magnetic' Gabor tight frame, we naturally derive the magnetic analogue of the Beals criterion originally considered by Iftimie-M{\u a}ntoiu-Purice.
\end{abstract}

\maketitle


\section{Introduction and main results}
\subsection{Introduction}

 The  Beals criterion \cite{B} naturally characterizes pseudo-differential operators  by their commutation properties with fundamental objects like  multiplication and differentiation operators; the basics of Weyl pseudo-differential calculus can be found in e.g. \cite{H-3}. To the best of our knowledge, all existent proofs of Beals' criterion use in an essential way some special properties of the Fourier transform and the translation invariance of the seminorms in $\mathscr{S}(\R^{2d})$, see for example Lemma 2.2 in \cite{Bo} or Proposition 8.2 in \cite{Di-Sj}. \\
In recent years it appeared useful to introduce a magnetic pseudo-differential calculus (see \cite{IMP1,IMP2,IP1} and references therein) which is adapted to the presence of long-range magnetic fields. The main motivation behind this particular class of operators was the need of  highlighting the magnetic flux effects and building up a gauge covariant calculus. Therefore, it was natural to search for a magnetic Beals-like criterion where the commutation with the plain momentum operators 
 should be replaced by the commutation with their magnetic counterparts. Such a criterion was indeed proved by Iftimie-M{\u a}ntoiu-Purice \cite{IMP2} and one of the technically heavy points in that work was the extension of Bony's lemma to the magnetic case. \\
 The main motivation of our paper is to propose an alternative proof of Beals' classical criterion based on the use of a normalized tight Gabor frame and to show how this approach can be  quite naturally extended to the magnetic case and recover the criterion established in \cite{IMP2}. Note that  no a-priori knowledge of the magnetic calculus is needed in order to understand the current manuscript.

\subsection{The non-magnetic case}

Let $X_j$ be the multiplication operator by $x_j$, $1\leq j\leq d$, while $D_j:=-i\partial_{x_j}$.  We introduce $W_k:=X_k$ when $1\leq k\leq d$ and $W_k:=D_{k-d}$ when $d+1\leq k\leq 2d$. The operators $W_k$ leave the Schwartz space $\mathscr{S}(\R^d)$ invariant. Let us consider a bounded map $T:\mathscr{S}(\R^d)\mapsto \mathscr{S}'(\R^d)$. 
Seen as maps from $\mathscr{S}(\R^d)$ to $\mathscr{S}'(\R^d)$, the following multiple commutators 
\begin{align}\label{1}
[W_{j_1},[W_{j_2},...,[W_{j_m},T]]...]\,,\quad m\geq 1, \quad  j_\ell \in\{1,2,..., 2d\}, 
\end{align}
are also bounded. Then the classical Beals criterion \cite{B} reads as follows:

\begin{theorem}\label{T-Beals}
Let us assume that both $T$ and all  possible commutators as in \eqref{1} can be extended to bounded operators on $L^2(\R^d)$. Then there exists a symbol $a_0(x,\xi)\in S^{0}_{0,0}(\R^{2d})$ such that for every $\Psi,\Phi\in \mathscr{S}(\R^d)$ we have:
$$\langle \Psi,T\Phi\rangle_{L^2(\mathbb{R}^d)}=(2\pi)^{-d}\int_{\R^{d}}\left (\int_{\R^{2d}} e^{i\xi \cdot (x-x')}\overline{\Psi(x)}a_0((x+x')/2,\xi)\Phi(x')dxdx' \right )d\xi, $$
 i.e. $T={\rm Op}^w(a_0)$ is the Weyl quantization of $a_0$.
\end{theorem}
In this paper, the scalar product of $L^2(\R^d)$ is linear in the second variable and we use the standard notation for H\"{o}rmander type symbols (see Section 7.8 in \cite{H-1}).

\subsection{The magnetic case}

 Let $d\geq 2$. Consider a $2$-form $B(x)=\sum_{1\leq j,k\leq d} B_{jk}(x)dx_j\wedge dx_k$   where  $B_{jk}=-B_{kj}$ are in $BC^\infty(\R^d)$ (i.e. the space of $C^\infty$ bounded functions together with all their derivatives). We assume that  the form is "magnetic" , i.e. that $\partial_j B_{k\ell } +\partial_k B_{\ell j} +\partial_\ell B_{jk}\equiv 0$ holds.  This simply expresses that the $2$-form is closed. Given any fixed $y\in\R^d$ one can construct a  $1$-form $A(\cdot,y)$ such that $B=dA(\cdot,y)$ and 
\begin{align}\label{9}
A_j(x,y)=-\sum_{k=1}^d\int_0^1 s\; (x_k-y_k)\; B_{jk}(y+s(x-y))ds.
\end{align}
We observe that $A_j(x,y)$ grows at most linearly in $|x-y|$, and this fact remains true for all its derivatives in $x$. Let $\Gamma_{y,x}$ denote the straight oriented segment joining $y$ with $x$. Since $A(\cdot,0)-A(\cdot,y)$ is closed and exact, we have the identity 
\begin{align}\label{10}
A_j(x,0)-A_j(x,y)=\partial_{x_j}\varphi(x,y),\quad 1\leq j\leq d\,,
\end{align}
\begin{equation}\label{defvarphi}
\varphi(x,y)=\int_{\Gamma_{y,x}}(A(\cdot,0)-A(\cdot,y))=\int_{\Gamma_{y,x}}A(\cdot,0)\,.
\end{equation}

 Here $A(\cdot,y)$ does not contribute to the integral because it is "orthogonal" to the integration path. The same orthogonality property allows us to identify $\varphi(x,y)$ with the circulation of $A(\cdot,0)$ on the oriented triangle generated by the origin, $y$ and $x$. Stokes' theorem implies that $\varphi(x,y)$  is equal with the magnetic flux through this triangle.

We denote the magnetic flux through the oriented triangle $\Delta(u,v,w)$ having vertices at $u,v,w\in \R^d$ as:
$$
 \Fl (u,v,w):=\int_{\Delta(u,v,w)}B\,,\qquad \Fl(x,y,0)=\varphi(x,y).
$$

We note the identities 
\begin{align}\label{11}
\varphi(x,y)=-\varphi(y,x)\quad {\rm and}\quad \varphi(u,v)+\varphi(v,w)-\varphi(u,w)=\Fl(u,v,w).
\end{align}

Now we can formulate the magnetic version of Beals' criterion as stated  in Theorem 1.1 of \cite{IMP2}. Let $\Pi_j:=D_j-A_j(\cdot,0)$ be the "magnetic" momenta  which also leave $\mathscr{S}(\R^d)$ invariant. We denote by $W_k$ either $X_k$ if $1\leq k\leq d$, or $\Pi_{k-d}$ if $d+1\leq k\leq 2d$. Let us consider a bounded map $T:\mathscr{S}(\R^d)\mapsto \mathscr{S}'(\R^d)$ and all possible commutators as in \eqref{1} but with the new $W_k$'s.  Here is the  magnetic Beals criterion:

\begin{theorem}\label{mainth}Let us assume that both $T$ and all the "magnetic" commutators as in \eqref{1} can be extended to bounded operators on $L^2(\R^d)$. Then there exists a symbol $a_0(x,\xi)\in S^{0}_{0,0}(\R^{2d})$ such that for every $\Psi,\Phi\in \mathscr{S}(\R^d)$ we have: 
$$\langle \Psi,T\Phi\rangle_{L^2(\mathbb{R}^d)}=(2\pi)^{-d}\int_{\R^{d}}\left (\int_{\R^{2d}} e^{i\varphi(x,x')}e^{i\xi \cdot (x-x')}\overline{\Psi(x)}a_0((x+x')/2,\xi)\Phi(x')dxdx' \right )d\xi\,. $$ 
\end{theorem}

\vspace{0.5cm}

 The rest of this manuscript is as follows: in Section \ref{sec2} we construct a family of tight frames which generalizes the classical Gabor case, in Section \ref{sec3} we give the proof of Theorem \ref{T-Beals},  and in Section \ref{sec4} we prove Theorem \ref{mainth}. 

\section{A magnetic normalized Gabor tight frame}\label{sec2}

Let $g\in C_0^\infty(\R^d;\mathbb{R})$  such that
${\rm supp}\,g\subset (-1,1)^d$ and  
\begin{align}\label{2}
g_\gamma(x):=g(x-\gamma),\quad \sum_{\gamma\in \Z^d}  g_\gamma^2(x) =1\,,\quad \forall x\in \R^d\,. 
\end{align}

 Let 
$\psi_m(x):= e^{i m\cdot x}$, for any $m\in \Z^d$. Denote by $\big(\tau_zf\big)(x):=f(x-z)$ the translation with $z\in\mathbb{R}^d$.  
\begin{lemma}\label{lema-gabor}
The functions 
\begin{equation}\label{eq:gabfra}
\big \{G^\varphi_{\gamma,m}(x):=(2\pi)^{-d/2}e^{i \varphi(x,\gamma)}(\tau_\gamma g\psi_m)(x):\; \gamma,m\in\Z^d\big \}\,,
\end{equation} 
with $\varphi$ defined in \eqref{defvarphi}, 
satisfy the identity
\beq\label{F-desc-ftest}
f(x)=\sum_{\gamma,m} G^\varphi_{\gamma,m}(x)\langle G^\varphi_{\gamma,m},f\rangle_{L^2(\mathbb{R}^d)}\,,\quad\forall f\in\mathscr{S}(\mathbb{R}^d)\,,
\eeq
where the series is absolutely convergent. In particular, these functions generate a normalized tight frame in $L^2(\R^d)$ (see \cite{Gr, Ch}). 
\end{lemma}
 
\begin{proof} 
We have
\begin{align}\label{horia1}
\langle G^\varphi_{\gamma,m},f\rangle_{L^2(\mathbb{R}^d)}& =
(2\pi)^{-d/2}\int_{\{\max_{j=1}^d |x_j-\gamma_j|\leq 1\}}\,g(x-\gamma)e^{- i m\cdot(x-\gamma)}e^{-i \varphi(x,\gamma)}f(x)\,dx \nonumber \\
& =(2\pi)^{-d/2}\int_{\{\max_{j=1}^d |x_j|\leq 1\}}\,e^{- i m\cdot x }g(x)\big(\tau_{-\gamma}e^{-i \varphi(\cdot,\gamma)}f\big)(x)\, dx  \nonumber\\ &  =:\mathcal{F}\big (g\tau_{-\gamma}e^{-i \varphi(\cdot,\gamma)}f\big )(m)
\end{align}
where $g\tau_{-\gamma}e^{-i \varphi(\cdot,\gamma)}f\in C^\infty_0\big((-1,1)^d\big)$ may be naturally considered, via its  $(2\pi\Z)^d$-periodic extension to $\R^d$,  as a function in $C^\infty (\mathbb{T}^d)$, and  where the right hand side of \eqref{horia1} is nothing but the $m$'th Fourier coefficient of $g\tau_{-\gamma}e^{-i \varphi(\cdot,\gamma)}f$. 

Integrating by parts in \eqref{horia1}, using \eqref{10} and the fact that $f$ is a Schwartz function, then given any $N\geq 1$ we may find a constant $C_{f,N}$  such that for every $m$ and $\gamma$ we have:
\begin{align}\label{hc1}
|\langle G^\varphi_{\gamma,m},f\rangle_{L^2(\mathbb{R}^d)}|\leq C_{f,N}\; <\gamma>^{-N} <m>^{-N} .
\end{align}
By the Fourier inversion formula and \eqref{horia1} we obtain:
$$
g\tau_{-\gamma}e^{-i \varphi(\cdot,\gamma)}f=(2\pi)^{-d/2}\underset{m\in\mathbb{Z}^d}{\sum}\psi_m \langle G^\varphi_{\gamma,m},f\rangle_{L^2(\mathbb{R}^d)}
$$
where the series converges absolutely due to \eqref{hc1} where we fix for example $N\geq 2d$. Translating by $\gamma\in\mathbb{Z}^d$ we obtain:
$$
g_\gamma(x) f(x)=(2\pi)^{-d/2}\underset{m\in\mathbb{Z}^d}{\sum}e^{i \varphi(x,\gamma)}(\tau_\gamma\psi_m)(x)\langle G^\varphi_{\gamma,m},f\rangle_{L^2(\mathbb{R}^d)},
$$
which coupled with \eqref{2} leads to:
$$ f=
\underset{\gamma\in\mathbb{Z}^d}{\sum}g_\gamma(g_\gamma f)=\underset{\gamma\in\mathbb{Z}^d}{\sum}\underset{m\in\mathbb{Z}^d}{\sum}G^\varphi_{\gamma,m}
\langle G^\varphi_{\gamma,m},f\rangle_{L^2(\mathbb{R}^d)}.
$$
This proves \eqref{F-desc-ftest}. 
\end{proof}

\section{Proof of Theorem \ref{T-Beals}}\label{sec3}

In order to simplify  notation, in the non-magnetic case $(\varphi\equiv 0$) we denote the Gabor frame by $G_{\gamma,m}$. By assumption, $T$ can be extended to a bounded operator on $L^2(\R^d)$ with norm $\|T\|$, thus:
\beq\label{est-Tmatrix-2}
\mathcal{T}_{\gamma,\gamma';m,m'}:= \big \langle G_{\gamma,m}, T G_{\gamma',m'}\big \rangle,\quad \big|\mathcal{T}_{\gamma,\gamma';m,m'}\big|\leq (2\pi)^{-d}\|g\|_{L^2(\mathbb{R}^d)}^2\,\|T\|\,.
\eeq

For every $N\in \mathbb N $, an application of the form
$$
\mathscr{S}(\mathbb{R}^d)\times\mathscr{S}(\mathbb{R}^d)\ni (\Phi,\Psi) \mapsto \underset{\max(|\gamma|,|m|,|\gamma'|,|m'|)\leq N}{\sum}\mathcal{T}_{\gamma,\gamma';m,m'}
\langle\Psi, G_{\gamma,m} \big\rangle_{L^2(\mathbb{R}^d)}\langle G_{\gamma',m'}, \Phi\big\rangle_{L^2(\mathbb{R}^d)}
$$
defines a tempered distribution on $\mathbb{R}^d\times\mathbb{R}^d$. Then the distribution kernel of the bounded operator $T:L^2(\mathbb{R}^d)\rightarrow L^2(\mathbb{R}^d)$ is given by the series
\begin{align}\label{5}
\mathring{T}=\underset{N\nearrow\infty}{\lim}\underset{\max(|\gamma|,|m|)\leq N}{\sum}\ \underset{\max(|\gamma'|,|m'|)\leq N}{\sum}\mathcal{T}_{\gamma,\gamma';m,m'} \big(G_{\gamma,m}\otimes \overline{G_{\gamma',m'}}\big),
\end{align}
where each finite sum belongs to $BC^\infty(\mathbb{R}^d\times\mathbb{R}^d)$  and converges weakly in the space of the tempered distributions on $\mathbb{R}^d\times\mathbb{R}^d$.  
If we restrict the distribution $\mathring{T}$ to a compact in $\mathbb R^d\times \mathbb R^d$ there exists a finite number of non-zero contributions from the series in $\gamma$ and $\gamma'$, but generally, the series in $m$ and $m'$ are not  absolutely  convergent. In order to remedy that difficulty, we make a regularization and define for $\varepsilon>0$:
\beq\label{F-Tepsilon}
\mathring{T}_\varepsilon(x,x'):=(2\pi)^{-d}\sum_{\gamma,\gamma'} \sum_{m,m'}\mathcal{T}_{\gamma,\gamma';m,m'} {g}_{\gamma}(x){g}_{\gamma'}(x') e^{-\varepsilon(|m|^2+|m'|^2)}e^{im\cdot (x-\gamma)}e^{-im'\cdot (x'-\gamma')}\,.
\eeq

Due to \eqref{est-Tmatrix-2}, it is not difficult to see that for a fixed $\varepsilon>0$ the function $\mathring{T}_\varepsilon$ is jointly continuous. We will later see that it is much more regular. We start by proving an estimate which is stronger than \eqref{est-Tmatrix-2}.
 
\begin{lemma}\label{lema-horia-1} 
Given any pair $N,M\geq 1$, there exists a constant $C_{N,M}$ such that
\begin{align}\label{6}
|\mathcal{T}_{\gamma,\gamma';m,m'}|\leq C_{N,M} <\gamma-\gamma'>^{-N} <m-m'>^{-M}.
\end{align}
\end{lemma}
\begin{proof}

The decay in $\gamma-\gamma'$ is a consequence of the fact that all commutators of $T$ with the position operators are bounded, see \eqref{1}.
 For example: 
$$
\begin{array}{ll}
(\gamma_1-\gamma'_1) \mathcal{T}_{\gamma,\gamma';m,m'}&= 
(2\pi)^{-d}\big \langle (\gamma_1-X_1) g_\gamma \psi_m(\cdot -\gamma), T g_{\gamma'} \psi_{m'}(\cdot -\gamma')\big \rangle\\&
\quad + (2\pi)^{-d}\big \langle g_\gamma \psi_m(\cdot -\gamma), [X_1,T]  g_{\gamma'} \psi_{m'}(\cdot -\gamma')\big \rangle\\&
\quad + (2\pi)^{-d}\big \langle g_\gamma \psi_m(\cdot -\gamma), T  (X_1-\gamma'_1) g_{\gamma'} \psi_{m'}(\cdot -\gamma')\big \rangle\,.
\end{array}
$$
The decay in $m-m'$ is due to the boundedness of the commutators of $T$ with the momentum operators (one  has to integrate by parts). For example: 
$$
\begin{array}{ll}
(m_1-m'_1) \mathcal{T}_{\gamma,\gamma';m,m'}&= 
(2\pi)^{-d}\big \langle m_1 g_\gamma \psi_m(\cdot -\gamma), T g_{\gamma'} \psi_{m'}(\cdot -\gamma')\big \rangle\\&
\qquad \qquad - (2\pi)^{-d}\big \langle g_\gamma \psi_m(\cdot -\gamma), T  m'_1 g_{\gamma'} \psi_{m'}(\cdot -\gamma')\big \rangle\\
&= 
(2\pi)^{-d}\big \langle g_\gamma (D_{x_1}\psi_m)(\cdot -\gamma), T g_{\gamma'} \psi_{m'}(\cdot -\gamma')\big \rangle\\&
\qquad \qquad - (2\pi)^{-d}\big \langle g_\gamma \psi_m(\cdot -\gamma), T   g_{\gamma'} (D_{x_1}\psi_{m'}) (\cdot -\gamma')\big \rangle\\
& = 
(2\pi)^{-d}\big \langle (D_{x_1} g)_\gamma \psi_m(\cdot -\gamma), T g_{\gamma'} \psi_{m'}(\cdot -\gamma')\big \rangle\\
&
\qquad - (2\pi)^{-d}\big \langle g_\gamma \psi_m(\cdot -\gamma), T (D_{x_1}g)_{\gamma'} \psi_{m'}(\cdot -\gamma')\big \rangle\\
& \qquad + (2\pi)^{-d}\big \langle g_\gamma \psi_m(\cdot -\gamma), [D_{x_1},T]  g_{\gamma'} \psi_{m'}(\cdot -\gamma')\big \rangle.\\
\end{array}
$$
\end{proof}

\vspace{0.2cm}

The next Lemma will show that the approximating kernel $\mathring{T}_\varepsilon$ has a fast off-diagonal decay. 
\begin{lemma}\label{lema-horia-2} 
Let $\varepsilon>0$. Then for every fixed $t\in \R^d$, the function
$$\R^d\ni s\mapsto \mathring{T}_\varepsilon(t+s/2,t-s/2)\in \C$$
belongs to $\mathscr{S}(\R^d)$.
\end{lemma}
\begin{proof}
Given $x$ and $x'$, the only $\gamma$'s and $\gamma'$'s contributing to \eqref{F-Tepsilon} must obey the conditions $|x-\gamma|\leq \sqrt{d}$ and $|x'-\gamma'|\leq \sqrt{d}$. Given $t=(x+x')/2$, the only $\gamma$'s and $\gamma'$'s contributing to \eqref{F-Tepsilon} must also obey $|(\gamma+\gamma')/2-t|\leq \sqrt{d}$. Thus:

\begin{align}\label{horia3}
\mathring{T}_\varepsilon(t+s/2,t-s/2)
=&\sum_{ |(\gamma+\gamma')/2-t|\leq \sqrt{d}}(2\pi)^{-d} \sum_{m,m'}\mathcal{T}_{\gamma,\gamma';m,m'} {g}(t+s/2-\gamma){g}(t-s/2-\gamma')\nonumber \\
&\qquad  \qquad\qquad \times  e^{-\varepsilon(|m|^2+|m'|^2)}e^{im\cdot (t+s/2-\gamma)}e^{-im'\cdot (t-s/2-\gamma')}\,.
\end{align}
The series in $m$ and $m'$ are absolutely convergent due to the regularizing Gaussians, while the sum in the  direction of $\gamma-\gamma'$ is convergent due to \eqref{6}. We can also differentiate as many times as we want with respect to $s$ in \eqref{horia3}, and the series remain absolutely convergent. 
Given $s=x-x'$, the only $\gamma$'s and $\gamma'$'s contributing to \eqref{horia3} must obey $|(\gamma-\gamma')-s|\leq 2\sqrt{d}$, thus when we estimate $s^\alpha D_s^\beta \mathring{T}_\varepsilon(t+s/2,t-s/2)$ we may replace $|s|$ with $|\gamma-\gamma'|+2\sqrt{d}$ and obtain something bounded (actually independent of $t$). More precisely,  given multi-indices $\alpha$ and $\beta$, there exists a constant $C(\alpha,\beta,\varepsilon)$ such that, for any $t\in \mathbb R$, 
$$\sup_{s\in\R^d}|s^\alpha D_s^\beta \mathring{T}_\varepsilon(t+s/2,t-s/2)|\leq C(\alpha,\beta,\varepsilon).$$
\end{proof}

Let us consider the  symbol associated by the Weyl quantization with the distribution kernel $\mathring{T}_\epsilon$:
\beq\label{F-distr-symbol}
a_\varepsilon(t,\xi):= \int_{\R^d} e^{-i\xi \cdot s} \mathring{T}_\varepsilon(t+s/2,t-s/2)\, ds.
\eeq
Due to Lemma \ref{lema-horia-2}, for fixed $t\in\mathbb{R}^d$  and $\varepsilon >0$, the function  $\xi \mapsto a_\varepsilon(t,\xi)$ is a Schwartz function.

\begin{lemma}\label{lema-horia-3}
The function $a_\varepsilon(t,\xi)$ converges uniformly on compact sets of $\R^{2d}$ to a smooth function $a_0(t,\xi)$. More precisely: 
$$\sup_{t\in \R^d}\sup_{\xi\in \R^d}|D_t^\alpha D_\xi^\beta a_\varepsilon(t,\xi)|\leq C(\alpha,\beta),\quad \forall \alpha,\beta\in \mathbb{N}^d,\quad \varepsilon\geq 0\,,$$
and  given any compact $K\subset \R^{2d}$ we have
$$\lim_{\varepsilon\searrow 0}\sup_{(t,\xi)\in K}|D_t^\alpha D_\xi^\beta\{a_\varepsilon(t,\xi)-a_0(t,\xi)\}|=0\,,\quad \forall \alpha,\beta\in \mathbb{N}^d.$$
In particular, $a_0\in S_{0,0}^0(\R^{2d})$. 
\end{lemma}

\noindent{\bf Remark}. Before proving the lemma, let us show how we can conclude the proof of Theorem \ref{T-Beals}. If $\Psi,\Phi\in C_0^\infty(\R^d)$ we have:
\begin{align*}
\langle \Psi, T\Phi\rangle &=\lim_{\varepsilon\searrow 0}\int_{\R^{2d}} \overline{\Psi(x)}\mathring{T}_\varepsilon(x,x')\Phi(x')dxdx'\\ 
&=(2\pi)^{-d}\lim_{\varepsilon\searrow 0} \int_{\R^{d}}\left (\int_{\R^{2d}} e^{i\xi \cdot (x-x')}\overline{\Psi(x)}a_\varepsilon((x+x')/2,\xi)\Phi(x')dxdx' \right )d\xi\\
&= (2\pi)^{-d}\int_{\R^{d}}\left (\int_{\R^{2d}} e^{i\xi \cdot (x-x')}\overline{\Psi(x)}a_0((x+x')/2,\xi)\Phi(x')dxdx' \right )d\xi\,,
\end{align*}
where the last equality follows from the Lebesgue dominated convergence theorem applied to the $\xi$ integral, for which we use Lemma \ref{lema-horia-3}.  Then the identity can be extended to $\mathscr{S}(\R^d)$ because $a_0\in S_{0,0}^0(\R^{2d})$. 

\vspace{0.2cm}

\noindent{ \it Proof of Lemma \ref{lema-horia-3}.} Let us introduce the notation 
$$
\kappa:=(\gamma+\gamma')/2\in\big(2^{-1}\mathbb{Z}\big)
^d,\ \kappa':=\gamma-\gamma'\in\mathbb{Z}^d,\ n:= (m+m')/2\in\big(2^{-1}\mathbb{Z}\big)^d,\ n':=m-m'\in\mathbb{Z}^d. 
$$
Using \eqref{F-Tepsilon} and \eqref{horia3} we obtain
\begin{align*}
a_\varepsilon(t,\xi)=&(2\pi)^{-d}\sum_{|\kappa-t|\leq \sqrt{d}}\sum_{\kappa'}
\sum_{n,n'} \int_{\R^d} e^{-i(\xi-n)\cdot  s}{g}(t-\kappa+(s-\kappa')/2){g}(t-\kappa-(s-\kappa')/2) ds 
\\
&\times e^{i n'\cdot t} e^{-i(n\cdot \kappa'+n'\cdot \kappa)}
e^{-\varepsilon(2|n|^2+|n'|^2/2)}\mathcal{T}_{\kappa,\kappa';n,n'}\,,
\end{align*}
 where in order to simplify notation  we write $\mathcal{T}_{\kappa,\kappa';n,n'}$ instead of  $\mathcal{T}_{\gamma,\gamma';m,m'}$. The estimate \eqref{6} insures a strong localization in both the $\kappa'$ and $n'$ series. The only series which apparently still needs $\varepsilon>0$ in order to converge, is the series in $n$. 

Define
\begin{align}\label{horia10}
F(t-\kappa,\xi-n,\kappa'):=(2\pi)^{-d}\int_{\R^d} e^{-i(\xi-n)\cdot  s}{g}(t-\kappa+(s-\kappa')/2){g}(t-\kappa-(s-\kappa')/2)  ds
\end{align}
so that 
\begin{align*}
a_\varepsilon(t,\xi)=&\sum_{|\kappa-t|\leq \sqrt{d}}\sum_{\kappa'}
\sum_{n,n'} e^{i n'\cdot t} F(t-\kappa,\xi-n,\kappa') e^{-i(n\cdot \kappa'+n'\kappa)}
e^{-\varepsilon(2|n|^2+|n'|^2/2)}\mathcal{T}_{\kappa,\kappa';n,n'}\,.
\end{align*}
It is important to remember that in the integral of \eqref{horia10}, the integrand is different from zero only if $s$ is of the order of $\kappa'$, i.e. $|s-\kappa'|\leq 2\sqrt{d}$. By differentiating $F$ with respect to $\xi$ we produce a polynomial growth in $s$ which can be traded off with a growth in $|\kappa'|$. Also, by standard partial integration with respect to $s$ we can generate a strong localization in $|\xi-n|$. In conclusion, one can prove the following statement: given any two multi-indices $\alpha,\beta\in \mathbb{N}^d$,  there exists a constant $C(\alpha,\beta)<\infty$ such that 
\begin{align}\label{horia11}
|D_t^\alpha D_\xi^\beta F(t-\kappa,\xi-n,\kappa')|\leq C(\alpha,\beta)\; <\xi-n>^{-2d} <\kappa'>^{|\beta|}.
\end{align}
The growth in $\kappa'$ is controlled by the decay of the matrix element $\mathcal{T}_{\kappa,\kappa';n,n'}$, while due to \eqref{horia11} the series in $n$ converges absolutely without any help from the $\varepsilon$-dependent Gaussian. 

Now we can take $\varepsilon$ to zero and define 
\begin{align*}
a_0(t,\xi):=&\sum_{|\kappa-t|\leq \sqrt{d}}\sum_{\kappa'}
\sum_{n,n'} e^{i n'\cdot t} F(t-\kappa,\xi-n,\kappa') e^{-i(n\cdot \kappa'+n'\kappa)} 
\mathcal{T}_{\kappa,\kappa';n,n'}\,.
\end{align*}
The limit can be taken uniformly on compacts in $\R^{2d}$, and remains valid for all possible derivatives with respect to both $\xi$ and $t$. 
This concludes the proof of the lemma.

\qed

\section{Proof of Theorem \ref{mainth}}\label{sec4}

 This time we let $\varphi\neq 0$ in \eqref{eq:gabfra} and  have 
$$G_{\gamma,m}^\varphi(x)=g(x-\gamma)e^{i\varphi(x,\gamma)} (2\pi)^{-d/2}\psi_m(x-\gamma),\quad\gamma,m\in\Z^d.$$ 
Then \eqref{5} reads as:
\begin{align}\label{13}
\mathring{T}(x,x')=\sum_{\gamma,\gamma'} \sum_{m,m'} {G}_{\gamma,m}^\varphi(x)\overline{{G}_{\gamma',m'}^\varphi(x')} \mathcal{T}_{\gamma,\gamma';m,m'}^\varphi,\quad 
\mathcal{T}_{\gamma,\gamma';m,m'}^\varphi:= \big \langle {G}_{\gamma,m}^\varphi, T {G}_{\gamma',m'}^\varphi\big \rangle.
\end{align}
Let us prove that $\mathcal{T}_{\gamma,\gamma';m,m'}^\varphi$ obeys exactly the same type of localization as in \eqref{6}. The localization in $\gamma-\gamma'$ follows just like before from the boundedness of commutators with the position operators, while the localization in $m-m'$ is obtained  
  by integration by parts and the use of the gauge covariance \eqref{10}.  For example, we have 
$$
\begin{array}{ll}
&(m_1-m'_1) \mathcal{T}_{\gamma,\gamma';m,m'}^\varphi= 
(2\pi)^{-d}\big \langle m_1 e^{i\varphi(\cdot ,\gamma)} g_\gamma \psi_m(\cdot -\gamma), T g_{\gamma'} e^{i\varphi(\cdot ,\gamma')} \psi_{m'}(\cdot -\gamma')\big \rangle\\&
\qquad \qquad - (2\pi)^{-d}\big \langle g_\gamma e^{i\varphi(\cdot ,\gamma)} \psi_m(\cdot -\gamma), T  m'_1 g_{\gamma'} e^{i\varphi(\cdot ,\gamma')} \psi_{m'}(\cdot -\gamma')\big \rangle\\
&= 
(2\pi)^{-d}\big \langle g_\gamma e^{i\varphi(\cdot ,\gamma)} (D_{x_1}\psi_m)(\cdot -\gamma), T g_{\gamma'} e^{i\varphi(\cdot ,\gamma')} \psi_{m'}(\cdot -\gamma')\big \rangle\\&
\qquad \qquad - (2\pi)^{-d}\big \langle e^{i\varphi(\cdot ,\gamma)} g_\gamma \psi_m(\cdot -\gamma), T   e^{i\varphi(\cdot ,\gamma')} g_{\gamma'} (D_{x_1}\psi_{m'}) (\cdot -\gamma')\big \rangle\\
& = 
(2\pi)^{-d}\big \langle e^{i\varphi(\cdot ,\gamma)} (D_{x_1} g)_\gamma \psi_m(\cdot -\gamma), T  e^{i\varphi(\cdot ,\gamma')} g_{\gamma'} \psi_{m'}(\cdot -\gamma')\big \rangle\\
&
\qquad - (2\pi)^{-d}\big \langle e^{i\varphi(\cdot ,\gamma)} g_\gamma \psi_m(\cdot -\gamma), T e^{i\varphi(\cdot ,\gamma')}(D_{x_1}g)_{\gamma'} \psi_{m'}(\cdot -\gamma')\big \rangle\\
& \qquad +(2\pi)^{-d}\big \langle e^{i\varphi(\cdot ,\gamma)} g_\gamma \psi_m(\cdot -\gamma), (-i\partial_{x_1}-\partial _{x_1}\varphi (\cdot ,\gamma))\, T \, e^{i\varphi(\cdot ,\gamma')} g_{\gamma'} \psi_{m'}(\cdot -\gamma')\big \rangle\\
& \qquad -(2\pi)^{-d}\big \langle e^{i\varphi(\cdot ,\gamma)} g_\gamma \psi_m(\cdot -\gamma), T (-i\partial_{x_1}-\partial_{x_1}\varphi (\cdot,\gamma')) e^{i\varphi(\cdot ,\gamma')} g_{\gamma'} \psi_{m'}(\cdot -\gamma')\big \rangle\\
& = 
(2\pi)^{-d}\big \langle e^{i\varphi(\cdot ,\gamma)} (D_{x_1} g)_\gamma \psi_m(\cdot -\gamma), T  e^{i\varphi(\cdot ,\gamma')} g_{\gamma'} \psi_{m'}(\cdot -\gamma')\big \rangle\\
&
\qquad - (2\pi)^{-d}\big \langle e^{i\varphi(\cdot ,\gamma)} g_\gamma \psi_m(\cdot -\gamma), T e^{i\varphi(\cdot ,\gamma')}(D_{x_1}g)_{\gamma'} \psi_{m'}(\cdot -\gamma')\big \rangle\\
& \qquad +(2\pi)^{-d}\big \langle e^{i\varphi(\cdot ,\gamma)} g_\gamma \psi_m(\cdot -\gamma), [D_{x_1}-A_1(\cdot,0),\, T] \, e^{i\varphi(\cdot ,\gamma')} g_{\gamma'} \psi_{m'}(\cdot -\gamma')\big \rangle\\
& \qquad - (2\pi)^{-d}\big \langle e^{i\varphi(\cdot ,\gamma)} A_1(\cdot ,\gamma)) g_\gamma \psi_m(\cdot -\gamma), \, T \, e^{i\varphi(\cdot ,\gamma')} g_{\gamma'} \psi_{m'}(\cdot -\gamma')\big \rangle\\
& \qquad +(2\pi)^{-d}\big \langle e^{i\varphi(\cdot ,\gamma)} g_\gamma \psi_m(\cdot -\gamma), T  e^{i\varphi(\cdot ,\gamma')} A_1 (x,\gamma') g_{\gamma'} \psi_{m'}(\cdot -\gamma')\big \rangle\,.
\end{array}
$$
Here the last formula was obtained by integration by parts.  We also used \eqref{10} for writing \break  $\partial_{x_1} \varphi (x,\gamma) = A_1(x,0) - A_1(x,\gamma)$ and the fact that on the support of ${g}_\gamma$ the function $A(\cdot,\gamma)$ is bounded uniformly in $\gamma$.

We now regularize the distributional kernel in \eqref{13} and introduce:

\begin{align}\label{14}
T_\varepsilon^\varphi \left (t+\frac{s}{2},t-\frac{s}{2}\right )&:=
(2\pi)^{-d}\nonumber  \sum_{\gamma,\gamma'} \sum_{m,m'} {g}_\gamma(t+s/2)e^{i\varphi(t+s/2,\gamma)}e^{-i\varphi(t-s/2,\gamma')}{g}_{\gamma'}(t-s/2)  \nonumber \\
&\qquad \times  e^{i(m-m')\cdot t} e^{i(m+m')\cdot s/2}  e^{-im\cdot\gamma}e^{im'\cdot \gamma'}e^{-\varepsilon(|m|^2+|m'|^2)}\mathcal{T}_{\gamma\, ,\gamma';m,m'}^\varphi.
\end{align}
This function has a rapid decay in $s$ and is smooth in both $t$ and $s$ when $\varepsilon>0$. Using twice the second identity of \eqref{11} we obtain: 

\begin{align}\label{15}
&\varphi(t+s/2,t-s/2)=\varphi(t+s/2,\gamma)+\varphi(\gamma,t-s/2)-{\Fl}(t+s/2,\gamma,t-s/2)\nonumber \\
&= \varphi(t+s/2,\gamma)+\varphi(\gamma,\gamma')+\varphi(\gamma',t-s/2)-{\Fl}(\gamma,\gamma',t-s/2)-{\Fl}(t+s/2,\gamma,t-s/2).
\end{align}

Let us introduce the quantity (we use \eqref{15} in the second equality):

\begin{align}\label{16}
a_\varepsilon(t,\xi):=& \int_{\R^d} e^{-i\varphi(t+s/2,t-s/2)}e^{-i\xi \cdot s} T_\varepsilon^\varphi(t+s/2,t-s/2)ds \\
=&(2\pi)^{-d}\sum_{|(\gamma +\gamma')/2-t|\leq \sqrt{d}} e^{-i\varphi(\gamma,\gamma')}  \nonumber \\
&\quad \times \sum_{m,m'} \int_{\R^d} e^{i {\Fl}(\gamma,\gamma',t-s/2)} e^{i{\Fl}(t+s/2,\gamma,t-s/2)} e^{-i[\xi-(m+m')/2]\cdot  s}{g}_\gamma(t+s/2){g}_{\gamma'}(t-s/2) ds \nonumber   \\
& \qquad \qquad \qquad   \times e^{i(m-m')\cdot t} e^{-im\cdot\gamma}e^{im'\cdot \gamma'}e^{-\varepsilon(|m|^2+|m'|^2)}\mathcal{T}_{\gamma,\gamma';m,m'}^\varphi.\nonumber 
\end{align}

As in the non-magnetic case, the only series which apparently poses convergence problems is the one with respect to the "direction" $(m+m')/2$. It turns out (as in the non-magnetic case) that the integral: 
$$\int_{\R^d} e^{i {\Fl}(\gamma,\gamma',t-s/2)} e^{i{\Fl}(t+s/2,\gamma,t-s/2)} e^{-i[\xi-(m+m')/2]\cdot  s}{g}_\gamma(t+s/2){g}_{\gamma'}(t-s/2) ds $$
is the one which insures decay in that direction. In order to prove it, let us notice that the fluxes ${\Fl}(t+s/2,\gamma,t-s/2)$ and ${\Fl}(\gamma,\gamma',t-s/2)$ grow like the area of the corresponding triangle, hence only like $|\gamma-\gamma'|$ because both $t+ s/2-\gamma$ and $t-s/2-\gamma'$ have a length of order one on the joint support of $g_\gamma$ and $g_{\gamma'}$; the same is true for their derivatives with respect to both $t$ and $s$. Integrating by parts with respect to $s$ we can generate a decay of the type $<\xi-(m+m')/2>^{-2d}$ at the price of a polynomial growth in $|\gamma-\gamma'|$, a growth which is taken care of by the decay of the matrix element $\mathcal{T}_{\gamma,\gamma';m,m'}^\varphi$. 

Thus the same strategy which was used in the previous section concerning the limit $\varepsilon\searrow 0$ can be repeated.  We conclude that $a_\varepsilon(t,\xi)\in S_{0,0}^0(\R^d)$ uniformly in $\varepsilon\geq 0$ and thus the symbol we are looking for is:
\begin{align}\label{17}
a_0(t,\xi)
=&\sum_{|(\gamma +\gamma')/2-t|\leq \sqrt{d}} e^{-i\varphi(\gamma,\gamma')}\nonumber \\
&  \quad  \times  \sum_{m,m'} \int_{\R^d} e^{i \Fl (\gamma,\gamma',t-s/2)} e^{i{\Fl}(t+s/2,\gamma,t-s/2)} e^{-i[\xi-(m+m')/2]\cdot  s}{g}_\gamma(t+s/2){g}_{\gamma'}(t-s/2) ds \nonumber \\
& \qquad \qquad \times e^{i(m-m')\cdot t} e^{-im\cdot\gamma}e^{im'\cdot \gamma'}\mathcal{T}_{\gamma,\gamma';m,m'}^\varphi.\nonumber 
\end{align}

\qed

\vspace{1cm}

\end{document}